\documentclass[11pt,hidelinks]{article}
\usepackage{psfrag,graphicx,amsmath,amsthm,amssymb,hyperref,color}

\newtheorem{theorem}{Theorem}[section]
\newtheorem{lemma}{Lemma}[section]
\numberwithin{equation}{section}

\newcommand{\aver}[1]{\left\{\!\!\left\{#1\right\}\!\!\right\}}

\newcommand{\jump}[1]{\left[\!\left[#1\right]\!\right]}
\newcommand{\bigjump}[1]{\left[\!\!\left[#1\right]\!\!\right]}

\DeclareFontEncoding{FMS}{}{}
\DeclareFontSubstitution{FMS}{futm}{m}{n}
\DeclareFontEncoding{FMX}{}{}
\DeclareFontSubstitution{FMX}{futm}{m}{n}
\DeclareSymbolFont{fouriersymbols}{FMS}{futm}{m}{n}
\DeclareSymbolFont{fourierlargesymbols}{FMX}{futm}{m}{n}
\DeclareMathDelimiter{\tbar}{\mathord}{fouriersymbols}{152}{fourierlargesymbols}{147}

\begin{document}
\baselineskip 11pt
\title{\Large\bf Discontinuous Galerkin Immersed Finite Element Methods \\
for Parabolic Interface Problems\thanks{Contract grant sponsor: Project of Shandong province higher educational science and technology program(J14LI03)}}
\author{Qing Yang\thanks{School of Mathematical Science, Shandong
Normal University, Jinan 250014, China}\ \ and Xu Zhang\thanks{Department of Mathematics,
Purdue University, West Lafayette, IN, 47907, xuzhang@purdue.edu}}
\date{}
\maketitle
\baselineskip 15pt
\begin{center}
\begin{minipage}{150mm}{\small
{\bf Abstract}\hspace{2mm} In this article, interior penalty discontinuous Galerkin methods using immersed finite
element functions are employed to solve parabolic interface problems.
Typical semi-discrete and fully discrete schemes are presented and analyzed. Optimal convergence for both semi-discrete and fully discrete schemes are proved. Some numerical experiments are provided to validate our theoretical results.
\par
 \vspace{3mm}
 {\bf Key words:}\hspace{2mm}  parabolic interface problems, discontinuous Galerkin,
 immersed finite element, error estimates.
\par
\vspace{3mm}
{\bf 2010 Mathematics Subject Classifications
}\hspace{2mm}65M15,\hspace{2mm}65M60}
\end{minipage}
\end{center}

\section{Introduction}
\par
\noindent
Let $\Omega\subset \mathbb{R}^2$ be a rectangular domain or a union of rectangular domains. Assume that $\Omega$ is separated by a smooth curve $\Gamma$ into two sub-domains $\Omega^-$ and $\Omega^+$, \emph{i.e.}, $\overline
\Omega=\overline{{\Omega^-}\cup {\Omega^+}\cup\Gamma}$, see the left plot in Figure \ref{fig: domain}. Let $[0,T]$ be a time interval.
We consider the linear parabolic interface problem
\begin{eqnarray}
\frac{\partial u}{\partial t}-\nabla\cdot(\beta\nabla u)&=&f(\mathbf{x},t),~~~~\mathbf{x} = (x,y)\in\Omega^+\cup\Omega^-,\ t\in  (0,T], \label{eq:parab_eq}\\
u&=&g(\mathbf{x},t),~~~~\mathbf{x}\in\partial\Omega,~~~t\in  (0,T], \label{eq:parab_eq_bc}\\
u&=&u_0(\mathbf{x}),~~~~~\mathbf{x}\in\overline{\Omega},~~~t=0. \label{eq:parab_eq_ic}
\end{eqnarray}
Here, the diffusion coefficient $\beta(\mathbf{x},t)$ is time independent and, without loss of generality, a piecewise constant function over $\Omega$, \emph{i.e.},
\[
\beta(\mathbf{x})=\left\{
\begin{array}{ll}
\beta^-,& ~~\mathbf{x} \in\Omega^-, \\
\beta^+,& ~~\mathbf{x} \in\Omega^+,
\end{array}
\right.
\]
and min$\{\beta^-,\beta^+\}>0$. Across the interface curve $\Gamma$, we assume that the solution and the normal component of the flux are continuous for any time $t\in [0,T]$, \emph{i.e.},
\begin{eqnarray}
 &&\jump{u}_{\Gamma}=0, \label{eq:parab_eq_jump_1} \\
 &&\bigjump{\beta\frac{\partial u}{\partial \mathbf{n}}}_{\Gamma}=0. \label{eq:parab_eq_jump_2}
\end{eqnarray}
Here $\jump{v}_{\Gamma} = (v|_{\Omega^+})|_{\Gamma} - (v|_{\Omega^-})|_{\Gamma} $ denotes the jump across the interface $\Gamma$.
\begin{figure}[ht]
\begin{minipage}[c]{0.5\textwidth}
\centering
\setlength{\unitlength}{1.5mm}
\includegraphics[totalheight=1.4in]{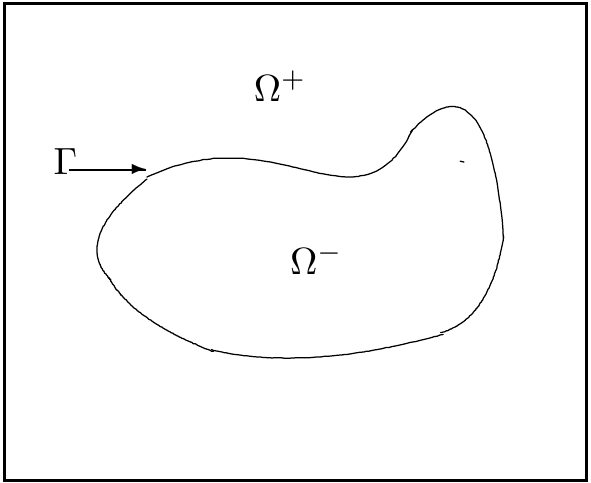}
\end{minipage}%
\begin{minipage}[c]{0.5\textwidth}
\centering
\setlength{\unitlength}{1.5mm}
\includegraphics[totalheight=1.4in]{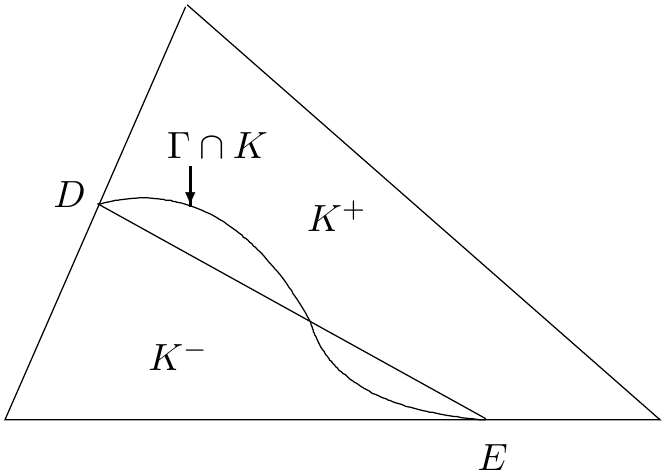}
\end{minipage}
\caption{\small The simulation domain $\Omega$ (left) and an interface element (right).}
\label{fig: domain}
\end{figure}
\par

In science and engineering, many physical phenomenons can be described by interface problems such as \eqref{eq:parab_eq} - \eqref{eq:parab_eq_jump_2}. Hence, solving interface problems accurately and efficiently is of great importance and has been studied for decades. It is well-known that classic numerical methods, such as finite element methods, use body-fitting meshes to solve interface problem to get optimal convergence \cite{2000BabuskaOsborn,1996BrambleKing,1998ChenZou}. The terminology body-fitting means a solution mesh has to be aligned with interfaces. Such restriction on mesh can hinder the applicability of conventional finite element methods in certain simulations. For example, it could prevent the use of structured meshes unless the interface geometry is trivial. In addition, when dealing with a moving interface problem, \emph{i.e.} $\beta = \beta(\mathbf{x},t)$, solution meshes have to be regenerated for each time step to be considered, which inevitably increases its computational costs. In order to overcome this limitation from the conventional finite element methods, the \textit{immersed finite element (IFE)} methods have been developed and extensively studied in the past two decades since the first article \cite{1998Li}. A prominent feature of IFE methods is that the solution mesh is independent of interface because IFE methods allow an interface to cut through elements, see the right plot in Figure \ref{fig: domain}. Consequently, one can use structured or even Cartesian meshes to solve problems with nontrivial interface geometry. This renders IFE methods great popularity in solving a variety of interface problems, such as elliptic interface problem \cite{2010ChouKwakWee,2008HeLinLin,2014JiChenLi,2003LiLinWu,2001LinLinRogersRyan,2015LinLinZhang,2010VallaghePapadopoulo,2013ZhangTHESIS}, elasticity interface problems \cite{2010GongLi,2013LinSheenZhang,2012LinZhang}, biharmonic interface problems \cite{2011LinLinSunWang}, and Stokes interface problems \cite{2015AdjeridChaabaneLin}, to name only a few.

So far, most IFE methods are developed for stationary interface problems. Recently, it starts to gain more attention on developing IFE methods for time-dependent interface problems. For instance, in \cite{2012WangWangYu}, transient
advection-diffusion equations with interfaces was treated by an immersed Eulerian-Lagrangian localized adjoint method. In \cite{2013LinSheen}, numerical solution to parabolic interface problem was considered by applying IFE methods together with the Laplacian transform. Crank-Nicolson-type fully discrete IFE methods and IFE method of lines were derived for parabolic problems with moving interface in \cite{2013HeLinLinZhang,2013LinLinZhang1}. Error analysis for a parabolic interface problem was presented in \cite{2015LinYangZhang2}.
\par
Discontinuous Galerkin (DG) finite element methods were introduced in 1970s \cite{1973BabuskaZlamal,1973ReedHill}. Because the discontinuous approximation functions are employed,  DG methods have many advantages such as high parallelizability, localizability, and easy handling of complicated geometries; therefore, DG methods have been used widely in solving different types PDEs \cite{2001ArnoldBrezziCockburnMarini,1999BaumannOden, 1976DouglasDupont,2008Riviere}. The idea of combining IFE and DG methods together to solve elliptic interface problems were proposed in \cite{2010HeLinLin, 2014HeLinLin}. Numerical analysis for discontinuous Galerkin immersed finite element (DG-IFE) methods was studied in our recent paper \cite{2015LinYangZhang1} for elliptic interface problem. The optimal convergence was obtained in a mesh-dependent energy norm. The aim of this paper is to extend the DG-IFE methods and error analysis for parabolic interface problem. One motivation to study the DG-IFE methods is that there is no continuity imposed on IFE space. Hence, it is more flexible perform local adaptive $h$ and $p-$ refinement, at the same time keeping solution meshes structured. This feature was demonstrated in \cite{2015LinYangZhang1} by various examples.

The rest of this paper is organized as follows. In Section 2, we consider the semi-discrete method and two prototypical fully discrete methods, \emph{i.e.}, backward Euler and Crank-Nicolson methods. In Sections 3 and 4, we derive the \textit{a priori} error estimates for semi-discrete and fully discrete methods, respectively. In Section 5, some numerical examples are reported to verify our theoretical estimates. A few concluding remarks are presented in Section 6.

\section{Discontinuous Galerkin immersed finite element methods}
In this section, we introduce the discontinuous Galerkin immersed finite element methods for solving the parabolic interface problem \eqref{eq:parab_eq} - \eqref{eq:parab_eq_jump_2}.
\subsection{Notations and Preliminaries}
Throughout this paper, we use standard notations for Sobolev spaces and their norms. In addition, we need to define piecewise Sobolev spaces which depend on the location of interface.
Let $D$ be a subset of $\Omega$ that is cut through by the interface $\Gamma$. For $r\geq 1$ , we define
\[
\tilde H^{r}(D)=\{v\in L^2(D): v|_{D\cap\Omega^s}\in
H^{r}(D\cap\Omega^s), s=+\ \text{or}\ -\}
\]
equipped with the norm
\[
\|v\|^2_{\tilde H^{r}(D)}:=\|v\|^2_{H^{r}(D\cap\Omega^-)}+\|v\|^2_{H^{r}(D\cap\Omega^+)}.
\]
For a function $u(\mathbf{x},t)$, we consider it as mapping from the time
interval $[0,T]$ to a normed space $V$ equipped with the norm $\|\cdot\|_{V}$. Furthermore, for any nonnegative number $k\geq 1$, we define
\[
L^k(0,T;V)=\left\{u: [0,T]\rightarrow V ~~\text{measurable}: \int_{0}^{T}\|u(\cdot,t)\|_V^kdt<\infty\right\},
\]
and
\[
\|u\|_{L^k(0,T;V)}=\left(\int_{0}^{T}\|u(\cdot,t)\|_V^kdt\right)^{1/k}.
\]
Similarly, we can define the standard space $H^p(0,T;V)$ for any integer $p>0$. Throughout this paper, we will use the letter $C$ to denote a generic positive constant which may take different values in different places. We usually use $u_t$
, $u_{tt}$, etc to denote the partial derivatives of $u$ with respect to the time variable $t$.

Let $\mathcal{T}_h = \{K\}$ be a Cartesian triangular or
rectangular mesh of $\Omega$ with mesh size $h$. An element $K$ is called an interface element if it is cut through by the interface $\Gamma$. Otherwise, we name it a non-interface element. The set of interface elements and non-interface elements of $\mathcal {T}_h$ are denoted by $\mathcal{T}_h^i$ and
$\mathcal{T}_h^n$, respectively.

Let $\mathcal{E}_h = \{e\}$ be the set of all edges in the mesh $\mathcal{T}_h$. Let $\mathcal {\mathring{E}}_h$ and $\mathcal{E}_h^b$ be the set of interior edges and boundary edges, respectively. Clearly, $\mathcal{E}_h = \mathcal {\mathring{E}}_h\cup\mathcal{E}_h^b$. An edge $e$ is called an interface edge if it intersect with $\Gamma$, otherwise it is a noninterface edge. The set of interface edges and non-interface edges are denoted by $\mathcal {E}_h^i$ and $\mathcal {E}_h^n$, respectively. Moreover, $\mathcal {\mathring{E}}_h^i$ and $\mathcal{\mathring{E}}_h^n$ denote the set of interior interface edges and interior non-interface edges, respectively.

Without loss of generality, we assume that the following hypotheses of mesh \cite{2015LinYangZhang1} hold:
\begin{description}
  \item[(H1)] If one edge of an element meets the interface $\Gamma$ at more than one point, this edge is part of $\Gamma$.
  \item[(H2)] If the interface $\Gamma$ meets the boundary of an element at two points, these two points are on
different edges of this element.
\end{description}
\par
According to conditions (H1) and (H2), each interface element intersects with the interface $\Gamma$ at two points, located on different edges. The intersection points are denoted by $D$ and $E$, and the line segment $\overline{DE}$ divides $K$ into two parts $K^+$ and $K^-$ such that
$K=K^+\cup K^-\cup \overline{DE}$, see
the right plot of Figure \ref{fig: domain}. We introduce the broken Sobolev space $\tilde H^2(\mathcal{T}_h)$ on the mesh $\mathcal{T}_h$:
\begin{eqnarray*}
\tilde H^2(\mathcal{T}_h)=\{v\in L^2(\Omega): && \forall
K\in\mathcal{T}_h^n, v|_K\in H^2(K);\\
 &&\forall
K\in\mathcal{T}_h^i,  v|_{K}\in H^1(K), v|_{K^s}\in H^2(K^s), s=+,-\}.
\end{eqnarray*}

\par
For each edge $e\in \mathcal{E}_h$, we assign a unit normal vector $\mathbf{n}_e$ according to the following rules: if $e\in\mathcal{E}_h^b$, then $\mathbf{n}_e$ is taken to be the unit outward normal vector of $\partial\Omega$; if $e\in\mathcal {\mathring E}_h$, shared by two elements $K_{e,1}$ and $K_{e,2}$, the normal vector $\mathbf{n}_e$ is oriented from $K_{e,1}$ to $K_{e,2}$. For a function
$u$ defined on $K_{e,1}\cup K_{e,2}$, which may be discontinuous across $e$, we define its \textit{average} and \textit{jump} as follows
\begin{equation}\label{eq: averjump}
  \aver{u}_e =
  \left\{
    \begin{array}{ll}
      \frac{1}{2}\left[(u|_{K_{e,1}})|_e+(u|_{K_{e,2}})|_e\right], &~ \text{if}~ e\in\mathcal {\mathring E}_h  \\
      u|_{e}, & ~\text{if}~e\in\mathcal{E}_h^b,
    \end{array}
  \right.~~~
  \jump{u}_e =
  \left\{
    \begin{array}{ll}
      (u|_{K_{e,1}})|_e-(u|_{K_{e,2}})|_e, & ~\text{if}~ e\in\mathcal {\mathring E}_h  \\
      u|_{e}, & ~\text{if}~e\in\mathcal{E}_h^b.
    \end{array}
  \right.
\end{equation}
For simplicity, we often drop the subscript $e$ from these notations as long as there is no danger to
cause any confusions.
\subsection{DG-IFE Methods}
In this subsection, we derive the DG-IFE methods for the parabolic interface problem
\eqref{eq:parab_eq} - \eqref{eq:parab_eq_jump_2}. First, we multiply equation \eqref{eq:parab_eq} by a test function $v\in
\tilde H^2(\mathcal{T}_h)$ and then integrate both sides on each element
$K\in\mathcal{T}_h$. For a non-interface element, a direct application of
 Green's formula gives
\begin{equation} \label{weak_local}
\int_Ku_t vd\mathbf{x}+\int_{K}\beta\nabla
u\cdot\nabla vd\mathbf{x}-\int_{\partial K}\beta\nabla
u\cdot{\mathbf{n}_K}vds=\int_K fvd\mathbf{x}.
\end{equation}
For an interface element, \eqref{weak_local} holds true as we perform Green's formula piecewisely on each sub-element separated by the interface.
For more detail of this procedure, we refer readers to \cite{2015LinYangZhang1}.
Then we summarize \eqref{weak_local} over all elements to obtain
 \begin{equation*}
\int_\Omega u_tvd\mathbf{x}+
\sum\limits_{K\in\mathcal{T}_h}\int_K\beta\nabla u\cdot \nabla
vd\mathbf{x}
-\sum_{K\in \mathcal{T}_h}\int_{\partial K}\beta\nabla u\cdot
\mathbf{n}_K vds=\int_{\Omega}fvd\mathbf{x}.
\end{equation*}
Rewriting the third term as the summation over all edges in $\mathcal{E}_h$, and using the notations in \eqref{eq: averjump} we have
 \begin{equation}
\int_\Omega u_tvd\mathbf{x}+\sum\limits_{K\in\mathcal{T}_h}\int_K\beta\nabla u\cdot \nabla
vd\mathbf{x}-\sum\limits_{e\in \mathcal{E}_h}\int_{e}\aver{\beta\nabla u\cdot
\mathbf{n}_e}\jump{v}ds=\int_{\Omega}fvd\mathbf{x}.
\end{equation}
Let $H_h=\tilde H^2(\Omega)+\tilde H^2(\mathcal{T}_h)$, then we can define a bilinear form $a_\epsilon$:
$H_h\times H_h\rightarrow \mathbb{R}$:
\begin{eqnarray}\nonumber
a_\epsilon(u,v)&=&\sum\limits_{K\in\mathcal{T}_h}\int_K\beta\nabla
u\cdot \nabla v d\mathbf{x}-\sum\limits_{e\in \mathcal{E}_h}\int_{e}\aver{\beta\nabla u\cdot \mathbf{n}_e}\jump{v}ds\\
&&+\epsilon\sum\limits_{e\in \mathcal{E}_h}\int_{e}\aver{\beta\nabla
v\cdot \mathbf{n}_e}\jump{u}ds+\sum\limits_{e\in
\mathcal{E}_h}\int_{e}\frac{\sigma_e}{|e|}\jump{u}\jump{v}ds,
\end{eqnarray}
where $\sigma_e\geq 0$ is the penalty parameter and $|e|$ stands for the length of
$e$. The parameter $\epsilon$ in $a_\epsilon(\cdot,\cdot)$ may take the value $-1$, $0$, or $1$. Note that $a_\epsilon(\cdot,\cdot)$ is symmetric if
$\epsilon=-1$ and is nonsymmetric otherwise. The regularity of exact solution $u\in H^1(\Omega)$ implies $\jump{u} =0$ on every interior edge $e\in \mathcal{\mathring E}_h$. Thus, for every $\epsilon$ we have
\[
\epsilon\sum\limits_{e\in \mathcal{\mathring E}_h}\int_{e}\aver{\beta\nabla v\cdot
\mathbf{n}_e}\jump{u}ds=0,~~~\text{and}~~~\sum\limits_{e\in
\mathcal{\mathring E}_h}\int_{e}\frac{\sigma_e}{|e|}\jump{u}\jump{v}ds=0.
\]
We define the linear form $L$: $H_h\rightarrow \mathbb{R}$:
\[
L(v)=\int_{\Omega}fvd\mathbf{x}+\sum\limits_{e\in \mathcal{E}_h^b}\int_{e}\left(\epsilon(\beta\nabla
v\cdot \mathbf{n}_e)+\frac{\sigma_e}{|e|}v\right)gds.
\]
Now, we obtain the weak form of the parabolic interface problem
\eqref{eq:parab_eq}-\eqref{eq:parab_eq_jump_2}:
\begin{eqnarray}
\left(u_t,
v\right)+a_\epsilon(u,v)&=&L(v), \label{eq:weak_form}\
\ \forall v\in \tilde H^2(\mathcal{T}_h),\\
u|_{t=0}&=&u_0. \label{eq:weak_init_cond}
\end{eqnarray}
\par
We now introduce finite-dimensional IFE subspaces of the broken Sobolev
space $\tilde H^2(\mathcal{T}_h)$, which will be used to approximate \eqref{eq:weak_form}-\eqref{eq:weak_init_cond}. For each element $K\in\mathcal{T}_h$, let $d_K=3$ for triangular elements and $d_K=4$ for rectangular elements.
If $K\in\mathcal{T}_h^n$, we choose $\phi_i(\mathbf{x})$, $1\leq i\leq d_K$ to be the standard linear or
bilinear nodal functions. Otherwise, $\phi_i(\mathbf{x})$, $1\leq i\leq d_K$ are chosen to be the linear
or bilinear IFE basis functions defined in \cite{2004LiLinLinRogers,2003LiLinWu} and \cite{2009HeTHESIS,2008HeLinLin}, respectively. For each element $K\in \mathcal{T}_h$, we define the local FE/IFE space to be
\[
S_h(K)=span\{\phi_i, 1\leq i\leq d_K\}.
\]
Then, the discontinuous IFE space over the
mesh $\mathcal{T}_h$ is defined as
\[
S_h(\mathcal{T}_h)=\{v\in L^2(\Omega):
v|_K\in S_h(K),~\forall K\in \mathcal{T}_h\}.
\]

\par
For every noninterface element $K \in \mathcal{T}_h^n$, $S_h(K)$ is a subspace of $H^2(K)$. For interface element $K \in \mathcal{T}_h^i$, every function $v\in S_h(K)$ is either a linear or a bilinear IFE function. It is has been shown in \cite{2004LiLinLinRogers} and \cite{2009HeTHESIS,2008HeLinLin} that such IFE function $v \in H^1(K)$ and $v|_{K^s} \in H^2(K^s), s = \pm$, but $v\notin H^2(K)$. It can be easily shown that $S_h(\mathcal{T}_h)\subset \tilde H^2(\mathcal{T}_h)$. We will use $S_h(\mathcal{T}_h)$ to discretize the weak formulation \eqref{eq:weak_form} and \eqref{eq:weak_init_cond} of the parabolic interface problem.\\

\noindent
{\bf Semi-discrete DG-IFE scheme}: Find $u_h: [0,T]\rightarrow
S_{h}(\mathcal{T}_h)$ such that
\begin{eqnarray}
\left( u_{h,t},
v_h\right)+a_\epsilon(u_h,v_h)&=&L(v_h),\ \ \forall v_h\in
S_{h}(\mathcal{T}_h), \label{eq:DG-IFE_semi}\\
u_h(\mathbf{x},0)&=&u_{0h}(\mathbf{x}),\ \ \ \mathbf{x}\in\Omega, \label{eq:DG-IFE_semi_ic}
\end{eqnarray}
where $u_{0h}$ is an approximation of $u_0$ in the space
$S_{h}(\mathcal{T}_h)$. \\

\noindent
 For a positive integer $N_t$, let $\Delta t=T/N_t$ be
the time step and $t^n=n\Delta t$, ($n=0,1,\cdots, N_t$). For any function $\varphi(t)$, we let $\varphi^n=\varphi(t^n)$, $n=0,1,\cdots,N_t$. For a sequence $\{\varphi^n\}_{n=0}^{N_t}$, we define
\[
\varphi^{n,\theta}=
\theta\varphi^{n}+(1-\theta)\varphi^{n-1}\ \ \forall~ 0\leq\theta\leq 1,\ \ ~~\partial_t \varphi^n=\frac{\varphi^{n}-\varphi^{n-1}}{\Delta t}, \ \ n=1,2,\cdots,N_t.
\]
{\bf Fully discrete DG-IFE scheme}: Find a sequence $\big\{u_h^n\big\}_{n=0}^{N_t}$ of
functions in $S_{h}(\mathcal{T}_h)$ such that
\begin{eqnarray}
\left(\partial_t u_h^n, v_h\right)+a_\epsilon(u_h^{n,\theta},v_h)&=&L^{n,\theta}(v_h),\ \ \forall v_h\in
S_{h}(\mathcal{T}_h), \label{eq:DG-IFE_full_disc}\\
u_h^0(\mathbf{x})&=&u_{0h}(\mathbf{x}),\ \ \ \mathbf{x}\in\Omega, \label{eq:DG-IFE_full_disc_ic}
\end{eqnarray}
where
\begin{equation*}
  L^{n,\theta}(v) = \int_{\Omega}f^{n,\theta}vd\mathbf{x}+\sum\limits_{e\in \mathcal{E}_h^b}\int_{e}\left(\epsilon(\beta\nabla
v\cdot \mathbf{n}_e)+\frac{\sigma_e}{|e|}v\right)g^{n,\theta}ds.
\end{equation*}
Note that the fully discrete DG-IFE scheme is the backward Euler scheme when $\theta=1$, and it is the
Crank-Nicolson scheme when $\theta=1/2$.

\section{Error Estimation for Semi-discrete Schemes}
\setcounter{section}{3}\setcounter{equation}{0}
\par
\noindent
In this section, we derive \textit{a priori} error estimates for semi-discrete scheme \eqref{eq:DG-IFE_semi} - \eqref{eq:DG-IFE_semi_ic}. The error bounds are based on the following mesh dependent energy norm:
\[
\tbar v\tbar=\left(\sum\limits_{K\in\mathcal{T}_h}\int_K\beta|\nabla
v|^2d\mathbf{x}+\sum\limits_{e\in \mathcal{
E}_h}\int_{e}\frac{\sigma_e}{|e|}\jump{v}^2ds\right)^{1/2},
\]
for all $v\in\tilde H^2(\mathcal{T}_h)$. We will first recall some results from \cite{2015LinYangZhang1} for elliptic problem.

\begin{lemma}
(Trace inequalities for IFE functions) Let $\mathcal{T}_h$ be a Cartesian triangular or rectangular mesh and let $K\in\mathcal{T}_h$ be an interface triangle or rectangle with diameter $h_K$ and let $e$ be an edge of $K$. There exists a constant $C$, independent of interface location but depending on the jump of the coefficient $\beta$, such that for every linear or bilinear IFE function $v$ defined on $K$, the following inequality holds:
\begin{equation}\label{eq:trace_ife_2}
  \|\beta \nabla v\cdot \mathbf{n}_e\|_{L^2(e)}\leq
Ch_K^{-1/2}\|\sqrt{\beta}\nabla v\|_{L^2(K)}.
\end{equation}
\end{lemma}

\begin{lemma} (Coercivity) There exists a constant $\kappa>0$ such that
 \begin{equation}
 a_\epsilon(v_h,v_h)\geq\kappa\tbar v_h\tbar^2,\ \ \forall~ v_h\in
 S_{h}(\mathcal{T}_h) \label{eq:coarcivity}
 \end{equation}
 holds for $\epsilon=1$ unconditionally and holds for
 $\epsilon=0$ or $\epsilon=-1$ when the penalty
 parameter $\sigma_e$ in $a_\epsilon(\cdot,\cdot)$ is large
 enough.
\end{lemma}
For every $t\in[0,T]$, we define the elliptic projection $P_hu$ of the exact solution
$u$ by
\begin{equation}
a_\epsilon(u-P_hu,v_h)=0,\ \ \forall v_h\in
S_h(\mathcal{T}_h). \label{eq:ellip_proj}
\end{equation}
It is easy to know that the solution to \eqref{eq:ellip_proj} exists and is unique. Moreover, it has the following error estimates.
\begin{lemma} (Estimate for elliptic projection) Assume that $u\in H^2(0,T;\tilde H^3(\Omega))$, for every $t\in [0,T]$, then the following error estimates hold
\begin{eqnarray}
  \tbar u-P_hu\tbar&\leq& Ch\|u\|_{\tilde H^3(\Omega)}, \label{eq:ell_proj_est_1} \\
  \tbar(u-P_hu)_t\tbar&\leq& Ch\|u_t\|_{\tilde H^3(\Omega)}, \label{eq:ell_proj_est_2}\\
  \tbar(u-P_hu)_{tt}\tbar&\leq& Ch\|u_{tt}\|_{\tilde
H^3(\Omega)}. \label{eq:ell_proj_est_3}
\end{eqnarray}
\end{lemma}
\begin{proof}
The estimate \eqref{eq:ell_proj_est_1} follows from the estimate derived for the DG-IFE methods for elliptic problems in \cite{2015LinYangZhang1}. Taking the time derivative of \eqref{eq:ellip_proj} we have
\[
0=\frac{d}{dt}a_\epsilon(u-P_hu,v_h)=a_\epsilon\left(u_t-(P_h
u)_t,v_h \right),\ \
\forall v_h\in S_h(\mathcal{T}_h),
\]
which implies  that $(P_hu)_t=P_hu_t$. Thus, for any $t\in [0,T]$, $u_t(\cdot,t)\in\tilde H^3(\Omega)$. Applying the estimate \eqref{eq:ell_proj_est_1} to $u_t$, we get
\[
\tbar(u-P_hu)_{t}\tbar=\tbar u_t-P_hu_t\tbar\leq Ch\|u_{t}\|_{\tilde
H^3(\Omega)}.
\]
This concludes the estimate \eqref{eq:ell_proj_est_2}. Following a similar argument, we can obtain \eqref{eq:ell_proj_est_3}.
\end{proof}

Now we are ready to derive an \textit{a priori} error estimate for the semi-discrete IFE scheme \eqref{eq:DG-IFE_semi} - \eqref{eq:DG-IFE_semi_ic}. First, we write
\begin{equation}\label{eq: xieta}
  u_h - u = (u_h - P_hu) + (P_hu - u) \triangleq \xi + \eta,
\end{equation}
where $P_hu$ is the elliptic projection of $u$ defined by \eqref{eq:ellip_proj}. From \eqref{eq:ell_proj_est_1}, we can bound $\tbar \eta\tbar$ as follows
 \begin{eqnarray}
 \tbar\eta\tbar&\leq& Ch\|u(\cdot,t)\|_{\tilde H^3(\Omega)}\leq Ch\Big(\|u_0\|_{\tilde H^3(\Omega)}+\|u_t\|_{L^2(0,T;\tilde H^3(\Omega)}\Big). \label{eq:semi_proof_1}
 \end{eqnarray}
It suffices to bound $\tbar\xi\tbar$. From \eqref{eq:weak_form}, \eqref{eq:DG-IFE_semi} and \eqref{eq:coarcivity}, we get the error equation for $\xi$
\begin{equation}
 \left(\xi_t, v_h\right)+a_\epsilon(\xi,v_h)=\left(\eta_t, v_h\right),\ \ \forall
 v_h\in S_{h}(\mathcal{T}_h). \label{eq:semi_proof_2}
 \end{equation}
Let $v_h=\xi_t$, then \eqref{eq:semi_proof_2} becomes
\begin{equation}
 \|\xi_t\|^2+a_\epsilon(\xi,\xi_t)
 =\left(\eta_t, \xi_t\right). \label{eq:semi_proof_3}
 \end{equation}
To proceed the analysis, we discuss the symmetric and nonsymmetric cases separately.\\

 \textbf{(i)}~If $\epsilon=-1$, then $a_\epsilon(\cdot,\cdot)$ is symmetric, and
\begin{equation}
 \left\|\xi_t\right\|^2+\frac{1}{2}\frac{d}{dt}a_\epsilon(\xi,\xi)
 \leq\|\eta_t\|\left\|\xi_t\right\|\leq \frac{1}{2}\left\|\eta_t\right\|^2+\frac{1}{2}\left\|\xi_t\right\|^2. \label{eq:semi_proof_4}
 \end{equation}
Note that $u_{h0} = P_hu_0$, thus $\xi(\cdot,0)=0$. We integrate both sides of \eqref{eq:semi_proof_4} from 0 to $t$ to obtain
\begin{equation}
\frac{1}{2}\int_0^t\left\|\xi_t\right\|^2d\tau+\frac{1}{2}a_\epsilon(\xi(\cdot,t),\xi(\cdot,t))
 \leq \frac{1}{2}\int_0^t\left\|\eta_t\right\|^2d\tau\leq Ch^2\int_0^T\left\|u_t\right\|_{\tilde H^3(\Omega)}^2dt. \label{eq:semi_proof_5}
\end{equation}
The second inequality in \eqref{eq:semi_proof_5} can be obtained from \eqref{eq:ell_proj_est_2}.
The coercivity of $a_\epsilon(\cdot,\cdot)$ leads to
\begin{equation}
\|\xi_t\|_{L^2(0,t;L^2(\Omega))}+\tbar\xi\tbar\leq Ch\|u_t\|_{L^2(0,T;\tilde H^3(\Omega))}. \label{eq:semi_proof_6}
\end{equation}
Dropping the first term in \eqref{eq:semi_proof_6} leads to a bound for $\tbar \xi\tbar$.\\

 \textbf{(ii)}~ If $\epsilon=1$ or $0$, then
$a_\epsilon(\cdot,\cdot)$ is nonsymmetric. We have
\begin{equation}\label{eq:semi_proof_7}
  a_\epsilon\left(\xi,\xi_t\right)=\frac{1}{2}\frac{d}{dt}a_\epsilon(\xi,\xi)
+\frac{1}{2}\left(a_\epsilon\left(\xi,\xi_t\right)-a_\epsilon\left(\xi_t,
\xi\right)\right) \geq\frac{1}{2}\frac{d}{dt}a_\epsilon(\xi,\xi)-\frac{C}{2}\left\tbar\xi_t\right\tbar^2-\frac{C}{2}\tbar\xi\tbar^2.
\end{equation}
Substituting \eqref{eq:semi_proof_7} into \eqref{eq:semi_proof_3} and integrating it from 0 to $t$, we have
\begin{equation}
 \int_0^t\|\xi_t\|^2d\tau+\tbar\xi\tbar^2
 \leq
 C\int_0^t(\|\eta_t\|^2+\tbar\xi_t\tbar^2+\tbar\xi\tbar^2)d\tau. \label{eq:semi_proof_8}
\end{equation}
Taking derivative of \eqref{eq:semi_proof_2} with respect to $t$ leads to
\begin{equation}
 \left(\xi_{tt}, v_h\right)+a_\epsilon(\xi_t,v_h)=\left(\eta_{tt}, v_h\right),~~~ \forall
 v_h\in S_{h}(\mathcal{T}_h). \label{eq:semi_proof_9}
 \end{equation}
Choosing $v_h=\xi_t$ in \eqref{eq:semi_proof_9} and using the coercivity of
$a_\epsilon(\cdot,\cdot)$, we get
\begin{eqnarray*}
 \frac{1}{2}\frac{d}{dt}\|\xi_{t}\|^2+\kappa \tbar\xi_t\tbar^2\leq
 \frac{1}{2}(\|\eta_{tt}\|^2+\|\xi_t\|^2).
\end{eqnarray*}
Integrating from 0 to $t$ and using the Gronwall inequality, we have
\begin{equation}
\int_0^t\tbar\xi_t\tbar^2d\tau\leq C\int_0^t\|\eta_{tt}\|^2d\tau+C\|\xi_t(\cdot,0)\|^2. \label{eq:semi_proof_10}
\end{equation}
Set $t=0$ and $v_h=\xi_t(\cdot,0)$ in \eqref{eq:semi_proof_2}, then we obtain
\begin{equation}
\|\xi_t(\cdot,0)\|\leq \|\eta_t(\cdot,0)\|. \label{eq:semi_proof_11}
\end{equation}
Substitute \eqref{eq:semi_proof_10} and \eqref{eq:semi_proof_11} into \eqref{eq:semi_proof_8} and apply the Gronwall inequality, then
\[
\int_0^t\|\xi_t\|^2d\tau+\tbar\xi\tbar^2
\leq C\int_0^t(\|\eta_t\|^2+\|\eta_{tt}\|^2)d\tau+C\|\eta_t(\cdot,0)\|^2.
\]
Applying the estimates \eqref{eq:ell_proj_est_2} and \eqref{eq:ell_proj_est_3} to the right hand side of the above inequality gives
\begin{equation}
\left\|\xi_t\right\|_{L^2(0,t;L^2(\Omega))}+\tbar\xi\tbar\leq Ch\Big(\|u_t(\cdot,0)\|_{\tilde
 H^3(\Omega))}+\|u_{t}\|_{L^2(0,T;\tilde
 H^3(\Omega))}+\|u_{tt}\|_{L^2(0,T;\tilde
 H^3(\Omega))}\Big). \label{eq:semi_proof_12}
\end{equation}
Again, dropping the first term leads to a bound for $\tbar\xi\tbar$. We summarize the above discussion in the following theorem.

\begin{theorem}
 Assume that the exact solution $u$ of problem \eqref{eq:parab_eq}-\eqref{eq:parab_eq_jump_2} satisfies $u\in H^1(0,T;\tilde H^3(\Omega))$ for $\epsilon=-1$ and
 $u\in H^2(0,T;\tilde H^3(\Omega))$ for $\epsilon=0, 1$, and $u_0\in\tilde H^3(\Omega)$.
 Let $u_h$ be the DG-IFE solution of \eqref{eq:DG-IFE_semi}-\eqref{eq:DG-IFE_semi_ic} and let $u_h(\cdot,0) = P_h u_0$ be the elliptic projection of $u_0$. Then there exists a constant $C$ such that for all $t\in [0,T]$
 \begin{equation}
\tbar u_h(\cdot,t)-u(\cdot,t)\tbar\leq Ch\Big(\|u_0\|_{\tilde H^3(\Omega)}+\|u_t\|_{L^2(0,T;\tilde H^3(\Omega))}\Big)
\label{eq:semi_est_1}
 \end{equation}
 for $\epsilon=-1$, and
  \begin{equation}
  \tbar u_h(\cdot,t)-u(\cdot,t)\tbar\leq Ch\Big(\|u_0\|_{\tilde H^3(\Omega)}+\|u_t(0)\|_{\tilde H^3(\Omega)}+\|u_t\|_{L^2(0,T;\tilde H^3(\Omega))}+\|u_{tt}\|_{L^2(0,T;\tilde H^3(\Omega))}\Big) \label{eq:semi_est_2}
 \end{equation}
 for $\epsilon=0$ or 1.
 \end{theorem}

\section{Error Estimation for Fully Discrete Schemes}
\noindent
Now we derive error estimates for the fully discrete DG-IFE schemes \eqref{eq:DG-IFE_full_disc} - \eqref{eq:DG-IFE_full_disc_ic}. We will consider two prototypical cases.

\noindent
\subsection{Backward Euler scheme}
 The backward Euler scheme corresponds to \eqref{eq:DG-IFE_full_disc} with $\theta=1$.
Subtracting \eqref{eq:weak_form} from \eqref{eq:DG-IFE_full_disc}, we can write the error equation as follows
\begin{eqnarray}
&&\left(\partial_t \xi^n,
v_h\right)+a_\epsilon(\xi^{n},v_h)=\left(\partial_t \eta^n,
v_h\right)+(r^n, v_h),\ \ \forall v_h\in
S_{h}(\mathcal{T}_h), \label{eq:Euler_1}
\end{eqnarray}
where $r^n=-(u_t^n-\partial_t u^n )$. Let $v_h=\partial_t\xi^n$ in \eqref{eq:Euler_1}, we obtain
\begin{equation}\label{eq:Euler_2}
  \|\partial_t\xi^n\|^2+ a_\epsilon(\xi^{n},\partial_t\xi^n)\leq\|\partial_t \eta^n\|^2+\|r^n\|^2+\frac{1}{2}\|\partial_t\xi^n\|^2.
\end{equation}
Again, the discussion for the second term is different for symmetric and nonsymmetric bilinear forms. We proceed in the following two cases.

\textbf{(i)} $\epsilon=-1$. The bilinear $a_\epsilon(\cdot,\cdot)$ is symmetric.
\begin{eqnarray*}
a_\epsilon(\xi^{n},\partial_t\xi^n)&=&\frac{1}{\Delta t}a_\epsilon(\xi^n,\xi^n-\xi^{n-1})\\
&=&\frac{1}{2\Delta t}\Big(a_\epsilon(\xi^n,\xi^n)-a_\epsilon(\xi^{n-1},\xi^{n-1})+a_\epsilon(\xi^n-\xi^{n-1},\xi^n-\xi^{n-1})\Big)\\
&\geq&\frac{1}{2\Delta t}\Big(a_\epsilon(\xi^n,\xi^n)-a_\epsilon(\xi^{n-1},\xi^{n-1})\Big).
\end{eqnarray*}
Then we substitute it into \eqref{eq:Euler_2} to get
\begin{eqnarray}
&&\frac{1}{2}\|\partial_t\xi^n\|^2+
\frac{1}{2\Delta t}\Big(a_\epsilon(\xi^n,\xi^{n})-a_\epsilon(\xi^{n-1},\xi^{n-1})\Big)\leq\|\partial_t \eta^n\|^2+\|r^n\|^2. \label{eq:Euler_3}
\end{eqnarray}
Multiplying \eqref{eq:Euler_3} by $2\Delta t$ and then summing over $n$ from $1$ to any positive integer $k$, we have
\begin{eqnarray}
&&\Delta t\sum\limits_{n=1}^{k}\|\partial_t\xi^n\|^2+
a_\epsilon(\xi^k,\xi^{k})\leq2\Delta t\sum\limits_{n=1}^k\Big(\|\partial_t \eta^n\|^2+\|r^n\|^2\Big). \label{eq:Euler_4}
\end{eqnarray}
Now we bound two terms on the right hand side of \eqref{eq:Euler_4}. By H\"{o}lder's inequality and \eqref{eq:ell_proj_est_2},
\begin{equation}\label{eq:Euler_5}
  \|\partial_t \eta^n\|^2= \int_{\Omega}\Big(\frac{1}{\Delta t}\int_{t^{n-1}}^{t^n}\eta_td\tau\Big)^2d\mathbf{x}
\leq\frac{1}{\Delta t}\int_{t^{n-1}}^{t^n}\|\eta_t\|^2d\tau\leq C\frac{h^2}{\Delta t}\int_{t^{n-1}}^{t^n}\|u_t\|_{\tilde H^3(\Omega)}^2d\tau,
\end{equation}
 \begin{equation}\label{eq:Euler_6}
 \|r^n\|^2=\int_{\Omega}|u_t^n-\partial_t u^n|^2d\mathbf{x}=\int_{\Omega}\left|\frac{1}{\Delta t}\int_{t^{n-1}}^{t^n}(t-t^{n-1})u_{tt}dt\right|^2d\mathbf{x}
\leq \frac{\Delta t}{3}\int_{t^{n-1}}^{t^n}\|u_{tt}\|^2d\tau.
 \end{equation}
Hence, by the coercivity of $a_\epsilon(\cdot,\cdot)$, we obtain
 \begin{eqnarray}
&&\Delta t\sum\limits_{n=1}^{k}\|\partial_t\xi^n\|^2+
\tbar\xi^k\tbar^2\leq C\Big(h^2\|u_t\|_{L^2(0,T;\tilde H^3(\Omega))}^2+(\Delta t)^2\|u_{tt}\|_{L^2(0,T;L^2(\Omega))}^2\Big). \label{eq:Euler_7}
\end{eqnarray}
Substituting \eqref{eq:semi_proof_1} and \eqref{eq:Euler_7} to $u_h^k-u^k= \xi^k+\eta^k$, and applying the triangle inequality yields
\begin{equation}
\tbar u_h^k-u^k\tbar\leq C\Big(h\big(\|u_0\|_{\tilde H^3(\Omega)}+\|u_t\|_{L^2(0,T;\tilde H^3(\Omega))}\big)+\Delta t\|u_{tt}\|_{L^2(0,T;L^2(\Omega))}\Big).
\label{eq:Euler_8}
\end{equation}
\par
\textbf{(ii)}~$\epsilon=0$ or 1. The bilinear form is nonsymmetric.
\begin{eqnarray*}
a_\epsilon(\xi^{n},\partial_t\xi^n)&=&\frac{1}{2\Delta t}\Big(a_\epsilon(\xi^n,\xi^n)-a_\epsilon(\xi^{n-1},\xi^{n-1})\Big)+\frac{1}{2\Delta t}a_\epsilon(\xi^n,\xi^n-\xi^{n-1})-\frac{1}{2\Delta t}a_\epsilon(\xi^n-\xi^{n-1},\xi^{n-1})\\
&=&\frac{1}{2\Delta t}\Big(a_\epsilon(\xi^n,\xi^n)-a_\epsilon(\xi^{n-1},\xi^{n-1})\Big)+\frac{1}{2}\Big(a_\epsilon(\partial_t\xi^n,\xi^{n})-a_\epsilon(\xi^{n-1},\partial_t\xi^{n})\Big)\\
&\geq&\frac{1}{2\Delta t}\Big(a_\epsilon(\xi^n,\xi^n)-a_\epsilon(\xi^{n-1},\xi^{n-1})\Big)-C\Big(\tbar\partial_t\xi^n\tbar^2+\tbar\xi^{n-1}\tbar^2+\tbar\xi^{n}\tbar^2\Big).
\end{eqnarray*}
Substituting it into \eqref{eq:Euler_2} leads to
\begin{eqnarray}
&&\frac{1}{2}\|\partial_t\xi^n\|^2+
\frac{1}{2\Delta t}\Big(a_\epsilon(\xi^n,\xi^{n})-a_\epsilon(\xi^{n-1},\xi^{n-1})\Big) \nonumber \\
&\leq&\|\partial_t \eta^n\|^2
+\|r^n\|^2+C\Big(\tbar\partial_t\xi^n\tbar^2+\tbar\xi^{n-1}\tbar^2+\tbar\xi^{n}\tbar^2\Big). \label{eq:Euler_9}
\end{eqnarray}
Multiplying \eqref{eq:Euler_9} by $2\Delta t$ and then summing over $n$, we obtain
\begin{equation}\label{eq:Euler_10}
\Delta t\sum\limits_{n=1}^k\|\partial_t\xi^n\|^2+
\kappa\tbar\xi^k\tbar^2\leq2\Delta t\sum\limits_{n=1}^k(\|\partial_t \eta^n\|^2
+\|r^n\|^2)+C\Delta t\sum\limits_{n=1}^k\tbar\partial_t\xi^n\tbar^2+C\Delta t\sum\limits_{n=1}^k\tbar\xi^{n}\tbar^2.
\end{equation}
From \eqref{eq:Euler_1} we have
\begin{eqnarray}
&&\frac{1}{\Delta t}\left(\partial_t \xi^n-\partial_t\xi^{n-1},
v_h\right)+a_\epsilon(v_h,\partial_t\xi^{n})=\left(\partial_{tt} \eta^n,
v_h\right)+(\partial_t r^n, v_h),\ \ \forall v_h\in
S_{h}(\mathcal{T}_h). \label{eq:Euler_11}
\end{eqnarray}
Let $v_h=\partial_t \xi^n$ in \eqref{eq:Euler_11}. Then
\begin{eqnarray*}
&&\frac{1}{2\Delta t}\left(\|\partial_t \xi^n\|^2-\|\partial_t\xi^{n-1}\|^2
\right)+\kappa\tbar\partial_t\xi^{n}\tbar^2\leq (\|\partial_{tt} \eta^n\|
+\|\partial_t r^n\|)\|\partial_t \xi^n\|.
\end{eqnarray*}
Multiplying the equation by $2\Delta t$, and taking summation over $n$, we obtain
\begin{eqnarray}
&&\|\partial_t \xi^k\|^2
+\Delta t\sum\limits_{n=2}^k\tbar\partial_t\xi^{n}\tbar^2\leq C\Delta t\sum\limits_{n=2}^k(\|\partial_{tt} \eta^n\|^2
+\|\partial_t r^n\|^2)+C\|\partial_t \xi^1\|^2. \label{eq:Euler_12}
\end{eqnarray}
Set $n=1$ and $v_h=\partial_t \xi^1=\xi^1/\Delta t$ in \eqref{eq:Euler_1}, then we have
\begin{eqnarray*}
&&\|\partial_t \xi^1\|^2+
\frac{1}{\Delta t}a_\epsilon(\xi^1,\xi^{1})\leq(\|\partial_t \eta^1\|+\|r^1\|)\|\partial_t \xi^1\|.
\end{eqnarray*}
Applying the coercivity of $a(\cdot,\cdot)$ and Young's inequality, we have
\begin{eqnarray*}
&&\|\partial_t \xi^1\|^2+\frac{1}{\Delta t}\tbar \xi^1\tbar^2\leq C(\|\partial_t \eta^1\|^2+\|r^1\|^2).
\end{eqnarray*}
Note that $\Delta t\tbar\partial_t\xi^{1}\tbar^2=\tbar\xi^{1}\tbar^2/\Delta t$. Substituting the above inequality into \eqref{eq:Euler_12} we have
\begin{eqnarray}
&&\sum\limits_{n=1}^k\Delta t\tbar\partial_t\xi^{n}\tbar^2\leq C\sum\limits_{n=2}^k\Delta t(\|\partial_{tt} \eta^n\|^2
+\|\partial_t r^n\|^2)+C(\|\partial_t \eta^1\|^2+\|r^1\|^2). \label{eq:Euler_13}
\end{eqnarray}
Substituting \eqref{eq:Euler_13} in \eqref{eq:Euler_10}, and applying the Gronwall inequality, we obtain
\begin{equation}\label{eq:Euler_14}
\tbar\xi^k\tbar^2\leq\sum\limits_{n=1}^k\Delta t(\|\partial_t \eta^n\|^2
+\|r^n\|^2)+C\sum\limits_{n=2}^k\Delta t(\|\partial_{tt} \eta^n\|^2
+\|\partial_t r^n\|^2)+C(\|\partial_t \eta^1\|^2+\|r^1\|^2).
\end{equation}
Now we bound the last four right-hand side terms in \eqref{eq:Euler_14}. First
\begin{eqnarray*}\nonumber
\|\partial_{tt} \eta^n\|^2&=&\int_{\Omega}\left(\frac{\eta^n-2\eta^{n-1}+\eta^{n-2}}{(\Delta t)^2}\right)^2d\mathbf{x}\\
&=&\int_{\Omega}\left(\frac{1}{(\Delta t)^2}\int_{t^{n-1}}^{t^n}\eta_{tt}(t^n-t)dt-\frac{1}{(\Delta t)^2}\int_{t^{n-2}}^{t^{n-1}}\eta_{tt}(t^{n-1}-t)dt\right)^2d\mathbf{x}\\
&\leq&\frac{1}{3\Delta t}\int_{t^{n-2}}^{t^n}\|\eta_{tt}\|^2dt.
\end{eqnarray*}
By \eqref{eq:ell_proj_est_3}, we have
\begin{eqnarray}
\Delta t\sum\limits_{n=2}^k\|\partial_{tt} \eta^n\|^2\leq Ch^2\|u_{tt}\|^2_{L^2(0,T;\tilde H^3(\Omega))}. \label{eq:Euler_15}
\end{eqnarray}
For the second term,
\begin{eqnarray*}\nonumber
\partial_{t}r^n&=&\frac{u_t^n-u_t^{n-1}}{\Delta t}-\frac{u^n-2u^{n-1}+u^{n-2}}{(\Delta t)^2}\\
&=&\int_{t^{n-1}}^{t^n}u_{ttt}dt-\frac{1}{(\Delta
t)^2}\int_{t^{n-1}}^{t^n}u_{ttt}(t^{n-1}-t)^2dt+\frac{1}{(\Delta
t)^2}\int_{t^{n-2}}^{t^{n-1}}u_{ttt}(t-t^{n-2})^2dt.
\end{eqnarray*}
Applying H\"{o}lder's inequality, we get
\begin{eqnarray}
\sum\limits_{n=2}^k\Delta t\|\partial_t r^n\|^2&\leq& (\Delta
t)^2\sum\limits_{n=2}^k\left(\int_{t^{n-1}}^{t^n}\|u_{ttt}\|^2dt+\frac{1}{5}\int_{t^{n-1}}^{t^n}\|u_{ttt}\|^2dt
+\frac{1}{5}\int_{t^{n-2}}^{t^{n-1}}\|u_{ttt}\|^2dt\right) \nonumber\\
&\leq& C(\Delta t)^2\|u_{ttt}\|^2_{L^2(0,T;L^2(\Omega))}. \label{eq:Euler_16}
\end{eqnarray}
As for the last two terms, we have
\begin{eqnarray}
\|\partial_t \eta^1\|^2\leq \frac{1}{\Delta t}\int_0^{\Delta
t}\|\eta_t\|^2dt\leq h^2\left(\frac{1}{\Delta t}\int_0^{\Delta
t}\|u_t\|_{\tilde H^3(\Omega)}^2dt\right) \label{eq:Euler_17}
\end{eqnarray}
and
 \begin{eqnarray}
 \|r^1\|^2=\int_{\Omega}|u_t^1-\partial_t u^1|^2d\mathbf{x}\leq \frac{(\Delta t)^2}{3}\left(\frac{1}{\Delta t}\int_{0}^{\Delta t}\|u_{tt}\|^2dt\right).
 \label{eq:Euler_18}
 \end{eqnarray}
 Now, substituting \eqref{eq:Euler_5}, \eqref{eq:Euler_6} and \eqref{eq:Euler_15}-\eqref{eq:Euler_18} into \eqref{eq:Euler_14}, we
 obtain
\begin{eqnarray*}\nonumber
\tbar\xi^k\tbar^2&\leq& Ch^2\left(\|u_t\|_{L^2(0,T;\tilde
H^3(\Omega))}^2+\|u_{tt}\|^2_{L^2(0,T;\tilde
H^3(\Omega))}+\frac{1}{\Delta t}\int_0^{\Delta t}\|u_t\|_{\tilde
H^3(\Omega)}^2dt\right)\\
&&+C(\Delta t)^2\left(\|u_{tt}\|_{L^2(0,T;L^2(\Omega))}^2+\|u_{ttt}\|^2_{L^2(0,T;L^2(\Omega))}+\frac{1}{\Delta
t}\int_{0}^{\Delta t}\|u_{tt}\|^2dt\right).
\end{eqnarray*}
Now, we summarize all the analysis above for the backward Euler DG-IFE method in the following theorem.
\begin{theorem}
Assume the exact solution $u$ of \eqref{eq:parab_eq}-\eqref{eq:parab_eq_jump_2} satisfies $u\in H^2(0,T;\tilde H^3(\Omega))\cap H^3(0,T;L^2(\Omega))$ and
 $u_0\in \tilde H^3(\Omega)$.
 Let the sequence $\left\{u_h^n\right\}_{n=0}^{N_t}$ be the solution of
the backward Euler scheme \eqref{eq:DG-IFE_full_disc}-\eqref{eq:DG-IFE_full_disc_ic}. Then, we have the following estimates satisfied for all $0\leq n\leq N_t$
\begin{description}
  \item[(1)] If $\epsilon=-1$, then there exists a positive constant $C$ independent of $h$ and $\Delta t$ such that
   \begin{equation}
\tbar u_h^n -u^n\tbar\leq C\Big(h\big(\|u_0\|_{\tilde H^3(\Omega)}+\|u_t\|_{L^2(0,T;\tilde H^3(\Omega))}\big)+\Delta t\|u_{tt}\|_{L^2(0,T;L^2(\Omega))}\Big). \label{eq:thm3.2-1}
\end{equation}
  \item[(2)] If $\epsilon=0$ or 1, then there exists a positive constant $C$ independent of $h$ and $\Delta t$ such that
  \begin{eqnarray}\nonumber
\tbar u_h^n -u^n\tbar
&\leq& Ch\left(\|u_0\|_{\tilde
H^3(\Omega)}+\|u_t\|_{L^2(0,T;\tilde
H^3(\Omega))}+\|u_{tt}\|_{L^2(0,T;\tilde
H^3(\Omega))}+\left(\frac{1}{\Delta t}\int_0^{\Delta
t}\|u_t\|_{\tilde
H^3(\Omega)}^2dt\right)^{1/2}\right)\\
&&+C\Delta t\left(\|u_{tt}\|_{L^2(0,T;L^2(\Omega))}+\|u_{ttt}\|_{L^2(0,T;L^2(\Omega))}+\left(\frac{1}{\Delta
t}\int_{0}^{\Delta t}\|u_{tt}\|^2dt\right)^{1/2}\right). \label{eq:thm3.2-2}
\end{eqnarray}
\end{description}
\end{theorem}

\noindent
\subsection{ Crank-Nicolson scheme}
Now we consider the error analysis for the Crank-Nicolson scheme corresponding to $\theta=1/2$ in \eqref{eq:DG-IFE_full_disc}. We only consider the symmetric case in which $\epsilon=-1$.
\\

 From \eqref{eq:weak_form}, \eqref{eq:DG-IFE_full_disc} and \eqref{eq:coarcivity}, we have
\begin{eqnarray}
&&\left(\partial_t \xi^n,
v_h\right)+\frac{1}{2}a_\epsilon(\xi^{n}+\xi^{n-1},v_h)=(\partial_t \eta^n,
v_h)+(r_1^n, v_h)+(r_2^n, v_h),\ \ \forall v_h\in
S_{h}(\mathcal{T}_h), \label{eq:CN_21}
\end{eqnarray}
where
\begin{eqnarray*}
&&r_1^n=u_t^{n-1/2}-\frac{1}{2}(u_t^n+u_t^{n-1}),~~r_2^n=-(u_t^{n-1/2}-\partial_t u^n).
\end{eqnarray*}
Taking $v_h=\partial_t \xi^n=(\xi^{n}-\xi^{n-1})/\Delta t$, we get
\begin{eqnarray}
\|\partial_t \xi^n\|^2+
\frac{1}{2\Delta t}a_\epsilon(\xi^{n}+\xi^{n-1},\xi^{n}-\xi^{n-1})&\leq&\Big(\|\partial_t \eta^n\|
+\|r_1^n\|+\|r_2^n\|\Big)\|\partial_t \xi^n\| \nonumber \\
&\leq&C\Big(\|\partial_t \eta^n\|^2
+\|r_1^n\|^2+\|r_2^n\|^2\Big)+\frac{1}{2}\|\partial_t \xi^n\|^2. \label{eq:CN_22}
\end{eqnarray}
Due to the symmetry of $a_{\epsilon}(\cdot,\cdot)$ (when $\epsilon=-1$) we have
\begin{eqnarray}\nonumber
\|\partial_t \xi^n\|^2+
\frac{1}{2\Delta t}\Big(a_\epsilon(\xi^{n},\xi^{n})-a_\epsilon(\xi^{n-1},\xi^{n-1})\Big)&\leq&\Big(\|\partial_t \eta^n\|
+\|r_1^n\|+\|r_2^n\|\Big)\|\partial_t \xi^n\|\\
&\leq&C\Big(\|\partial_t \eta^n\|^2
+\|r_1^n\|^2+\|r_2^n\|^2\Big). \label{eq:CN_23}
\end{eqnarray}
Multiplying \eqref{eq:CN_23} by $2\Delta t$ and summing over $n$, we have
\begin{eqnarray}
\kappa\tbar\xi^k\tbar^2\leq a_\epsilon(\xi^{k},\xi^{k})&\leq&C\sum\limits_{n=1}^k\Delta t\Big(\|\partial_t \eta^n\|^2
+\|r_1^n\|^2+\|r_2^n\|^2\Big). \label{eq:CN_24}
\end{eqnarray}
Note that \eqref{eq:Euler_5} provides a bound for $\|\partial_t \eta^n\|^2$, hence we only need to estimate $\|r_1^n\|^2$ and $\|r_2^n\|^2$. Applying Taylor formula and H\"{o}lder's inequality, we obtain
\begin{eqnarray}\nonumber
\|r_1^n\|^2&=&\int_{\Omega}\Big(u_t^{n-1/2}-\frac{1}{2}(u_t^n+u_t^{n-1})\Big)^2d\mathbf{x}\\ \nonumber
&=&\int_{\Omega}\frac{1}{4}\left(\int_{t^{n-1}}^{t^{n-1/2}}u_{ttt}(t-t^{n-1})dt+\int_{t^{n-1/2}}^{t^{n}}u_{ttt}(t^n-t)dt\right)^2d\mathbf{x}\\
&\leq& C(\Delta t)^3\int_{t^{n-1}}^{t^{n}}\|u_{ttt}\|^2dt, \label{eq:CN_25}
\end{eqnarray}
and similarly
\begin{equation}\label{eq:CN_26}
  \|r_2^n\|^2 \leq C(\Delta t)^3\int_{t^{n-1}}^{t^{n}}\|u_{ttt}\|^2dt
\end{equation}
Put \eqref{eq:Euler_5}, \eqref{eq:CN_25} and \eqref{eq:CN_26} in \eqref{eq:CN_24} then we have
\[
\tbar\xi^k\tbar^2\leq C\Big(h^2\|u_t\|^2_{L^2(0,T;\tilde H^3(\Omega))}+(\Delta t)^4\|u_{ttt}\|^2_{L^2(0,T;L^2(\Omega))}\Big).
\]
Now we summarize the result in the following theorem.
\begin{theorem}
Assume that $u\in H^1(0,T;\tilde H^3(\Omega))\cap H^3(0,T;L^2(\Omega))$ is a solution to the interface problem \eqref{eq:parab_eq}-\eqref{eq:parab_eq_jump_2} and $u_0\in \tilde H^3(\Omega)$. Assume $\left\{u_h^n\right\}_{n=0}^{N_t}$ is the solution of Crank-Nicolson scheme \eqref{eq:DG-IFE_full_disc}-\eqref{eq:DG-IFE_full_disc_ic} with $\epsilon=-1$. Then, there exists a positive constant $C$ independent of $h$ and $\Delta t$ such that for all $0\leq n\leq N_t$
\begin{equation}\label{eq: thm3.3}
\tbar u_h^n -u^n\tbar\leq C\Big(h(\|u_0\|_{\tilde
H^3(\Omega)}+\|u_t\|_{L^2(0,T;\tilde
H^3(\Omega))})+(\Delta t)^2\|u_{ttt}\|_{L^2(0,T;L^2(\Omega))}\Big).
\end{equation}
\end{theorem}

\section{Numerical Examples}
\setcounter{equation}{0}

\noindent
In this section, we report some numerical results of DG-IFE methods for parabolic interface problems. Let the solution domain be $\Omega\times[0,T]$, where $\Omega$ is the unit square $(0,1)\times(0,1)$, and $T=1$. The interface curve $\Gamma$ is an ellipse centered at the point $(x_0,y_0)$ with semi-radius $a$ and $b$. The parametric form is given by
\begin{equation}\label{eq: ellipse parametric form}
  \left\{
    \begin{array}{ll}
      x = x_0 + a\cos(\theta), \\
      y = y_0 + b\sin(\theta).
    \end{array}
  \right.
\end{equation}
In our computation, we choose $x_0 = y_0 = 0$, $a = \pi/4$, $b=\pi/6$, and we consider the first quadrant of the ellipse as the interface, \emph{i.e.}, $\theta\in [0,\frac{\pi}{2}]$. Note that the interface curve $\Gamma$ touches the boundary of $\Omega$ and separates $\Omega$ into two sub-domains denoted by

\begin{equation*}
  \Omega^- = \{\mathbf{x}:r(\mathbf{x})<1\},~~~\text{and}~~~
  \Omega^+= \{\mathbf{x}:r(\mathbf{x})>1\}
\end{equation*}
where
\begin{equation*}
  r(\mathbf{x})=r(x,y) = \sqrt{\frac{(x-x_0)^2}{a^2} + \frac{(y-y_0)^2}{b^2}}.
\end{equation*}
The source function $f$ and the boundary function $g$ in the parabolic interface problem are chosen such that the exact solution $u$ is as follows
\begin{equation}\label{eq: true solution}
    u(\mathbf{x},t) =
    \left\{
      \begin{array}{ll}
        \frac{1}{\beta^-}r(\mathbf{x})^p e^{t}, & ~\text{if~} \mathbf{x}\in \Omega^- , \\
        \left(\frac{1}{\beta^+}r(\mathbf{x})^p - \frac{1}{\beta^+} + \frac{1}{\beta^-}\right) e^{t}, & ~\text{if~} \mathbf{x}\in \Omega^+,
      \end{array}
    \right.
\end{equation}
where $p = 5$ and the coefficients $\beta^{\pm}$ vary in different examples.

We use Cartesian rectangular meshes $\mathcal{T}_h, h>0$ formed by partitioning $\Omega$ into $N_s\times N_s$ congruent rectangles of size $h = 1/N_s$ for a set of integers $N_s$. For the fully discretization, we divide the time interval $[0,T]$ uniformly into $N_t$ subintervals with $t_n = n\Delta t$, $n=0,1,\cdots, N_t$, and $\Delta t = T/N_t$.

\subsubsection*{Example 1: Moderate Jump $(\beta^-,\beta^+) = (1,10)$}
First we choose diffusion coefficient $(\beta^-,\beta^+) = (1,10)$ which represents a moderate discontinuity across the interface. Both the nonsymmetric and symmetric DG-IFE schemes are employed to solve the elliptic interface problem at each time level. We choose the penalty parameters $\sigma_e = 100$ for symmetric DG-IFE scheme and $\sigma_e = 1$ for nonsymmetric DG-IFE scheme. Backward Euler and Crank-Nicolson schemes are used for fully discretization. Errors of IFE solutions in $L^\infty$, $L^2$, and semi-$H^1$ norms are computed at the final time level, \emph{i.e.}, $t=1$. Data listed in Table \ref{table: nonsym DGIFE error 1 10} and Table \ref{table: sym DGIFE error 1 10} are generated with time step size $\Delta t = 2h$.

\begin{table}[ht]
\begin{center}
\begin{footnotesize}
\begin{tabular}{|c|cc||cc||cc||cc|cc||cc|}
\hline
& \multicolumn{6}{|c||}{\textbf{Backward Euler}}&  \multicolumn{6}{|c|}{\textbf{Crank Nicolson}}\\
\hline
$N_s$
& $\|\cdot\|_{L^\infty}$ & rate
& $\|\cdot\|_{L^2}$ & rate
& $|\cdot|_{H^1}$ & rate
& $\|\cdot\|_{L^\infty}$ & rate& $\|\cdot\|_{L^2}$ & rate
& $|\cdot|_{H^1}$ & rate  \\
 \hline
$10$      &$1.85E{-1}$ &         &$1.37E{-1}$ &         &$2.16E{-0}$ &
          &$2.43E{-1}$ &         &$1.59E{-1}$ &         &$2.21E{-0}$ &      \\
$20$      &$5.06E{-2}$ & 1.87    &$3.52E{-2}$ & 1.96    &$1.08E{-0}$ & 1.01
          &$5.00E{-2}$ & 2.28    &$3.65E{-2}$ & 2.13    &$1.07E{-0}$ & 1.04 \\
$40$      &$1.38E{-2}$ & 1.87    &$9.16E{-3}$ & 1.94    &$5.41E{-1}$ & 0.99
          &$1.34E{-2}$ & 1.90    &$9.43E{-3}$ & 1.95    &$5.40E{-1}$ & 0.99 \\
$80$      &$3.77E{-3}$ & 1.87    &$2.41E{-3}$ & 1.93    &$2.71E{-1}$ & 1.00
          &$3.43E{-3}$ & 1.96    &$2.36E{-3}$ & 2.00    &$2.70E{-1}$ & 1.00 \\
$160$     &$1.05E{-3}$ & 1.84    &$6.65E{-4}$ & 1.86    &$1.36E{-1}$ & 1.00
          &$8.59E{-4}$ & 2.00    &$5.92E{-4}$ & 2.00    &$1.36E{-1}$ & 1.00 \\
$320$     &$3.21E{-4}$ & 1.71    &$1.97E{-4}$ & 1.75    &$6.80E{-2}$ & 1.00
          &$2.18E{-4}$ & 1.98    &$1.48E{-4}$ & 2.00    &$6.80E{-2}$ & 1.00 \\
\hline
\end{tabular}
\caption{Errors of nonsymmetric DG-IFE solutions with $\beta^- = 1$, $\beta^+ = 10$}
\label{table: nonsym DGIFE error 1 10}
\end{footnotesize}
\end{center}
\end{table}

\begin{table}[ht]
\begin{center}
\begin{footnotesize}
\begin{tabular}{|c|cc||cc||cc||cc|cc||cc|}
\hline
& \multicolumn{6}{|c||}{\textbf{Backward Euler}}&  \multicolumn{6}{|c|}{\textbf{Crank Nicolson}}\\
\hline
$N_s$
& $\|\cdot\|_{L^\infty}$ & rate
& $\|\cdot\|_{L^2}$ & rate
& $|\cdot|_{H^1}$ & rate
& $\|\cdot\|_{L^\infty}$ & rate& $\|\cdot\|_{L^2}$ & rate
& $|\cdot|_{H^1}$ & rate  \\
 \hline
$10$      &$8.91E{-2}$ &         &$6.19E{-2}$ &         &$2.11E{-0}$ &
          &$1.11E{-1}$ &         &$6.63E{-2}$ &         &$2.12E{-0}$ &      \\
$20$      &$2.56E{-2}$ & 1.80    &$1.523{-2}$ & 2.02    &$1.07E{-0}$ & 0.98
          &$1.63E{-2}$ & 2.77    &$1.70E{-2}$ & 1.96    &$1.07E{-0}$ & 0.99 \\
$40$      &$6.91E{-3}$ & 1.89    &$3.85E{-3}$ & 1.98    &$5.40E{-1}$ & 0.99
          &$4.65E{-3}$ & 1.81    &$4.13E{-3}$ & 2.05    &$5.39E{-1}$ & 0.98 \\
$80$      &$1.91E{-3}$ & 1.86    &$1.04E{-3}$ & 1.89    &$2.71E{-1}$ & 0.99
          &$1.33E{-3}$ & 1.80    &$1.07E{-3}$ & 2.02    &$2.71E{-1}$ & 0.99 \\
$160$     &$5.15E{-4}$ & 1.89    &$3.06E{-4}$ & 1.76    &$1.36E{-1}$ & 1.00
          &$3.97E{-4}$ & 1.75    &$2.52E{-4}$ & 2.01    &$1.36E{-1}$ & 1.00 \\
$320$     &$1.35E{-4}$ & 1.92    &$1.04E{-4}$ & 1.57    &$6.79E{-1}$ & 1.00
          &$1.17E{-4}$ & 1.76    &$6.27E{-5}$ & 2.01    &$6.79E{-1}$ & 1.00 \\
\hline
\end{tabular}
\caption{Errors of symmetric DG-IFE solutions with $\beta^- = 1$, $\beta^+ = 10$}
\label{table: sym DGIFE error 1 10}
\end{footnotesize}
\end{center}
\end{table}
In Table \ref{table: nonsym DGIFE error 1 10} and Table \ref{table: sym DGIFE error 1 10}, errors in semi-$H^1$ norm, which is equivalent to energy norm, have optimal convergence rate $O(h)$ for both nonsymmetric and symmetric DG-IFE schemes. These results confirm our theoretical error analysis \eqref{eq:thm3.2-1} and \eqref{eq:thm3.2-2} for backward Euler error estimation and \eqref{eq: thm3.3} for Crank-Nicolson error estimation. We also note that convergence rate of errors of Crank-Nicolson solutions in $L^2$ norm are $O(h^2)$, although we do not have the corresponding theoretical analysis yet. For backward Euler, the $L^2$ convergence rate is decreasing from $O(h^2)$ to $O(h)$ as we perform uniform mesh refinement. Because for small $h$, error in time discretization dominates, which has only the first order.

\subsubsection*{Example 2: Flipped Coefficient $(\beta^-,\beta^+) = (10,1)$}
In this example we test the robustness of the algorithm by flipping the diffusion coefficient such that  $(\beta^-,\beta^+) = (10,1)$. This represents a change of the material property. Again, we use both nonsymmetric and symmetric DG-IFE schemes. The penalty parameters are chosen as $\sigma_e = 100$ for symmetric DG-IFE scheme and $\sigma_e = 1$ for nonsymmetric DG-IFE scheme. Errors of IFE solutions are computed at the final time level, \emph{i.e.}, $t=1$, and are reported in Table \ref{table: nonsym DGIFE error 10 1} and Table \ref{table: sym DGIFE error 10 1}. We can see that the pattern of error decay are similar to the first example.

\begin{table}[ht]
\begin{center}
\begin{footnotesize}
\begin{tabular}{|c|cc||cc||cc||cc|cc||cc|}
\hline
& \multicolumn{6}{|c||}{\textbf{Backward Euler}}&  \multicolumn{6}{|c|}{\textbf{Crank Nicolson}}\\
\hline
$N_s$
& $\|\cdot\|_{L^\infty}$ & rate
& $\|\cdot\|_{L^2}$ & rate
& $|\cdot|_{H^1}$ & rate
& $\|\cdot\|_{L^\infty}$ & rate& $\|\cdot\|_{L^2}$ & rate
& $|\cdot|_{H^1}$ & rate  \\
 \hline
$10$      &$9.37E{-1}$ &         &$8.37E{-1}$ &         &$1.82E{+1}$ &
          &$1.19E{-0}$ &         &$8.66E{-1}$ &         &$1.83E{+1}$ &      \\
$20$      &$2.64E{-1}$ & 1.83    &$2.29E{-1}$ & 1.87    &$9.08E{-0}$ & 1.00
          &$1.87E{-1}$ & 2.67    &$2.23E{-1}$ & 1.96    &$9.06E{-0}$ & 1.01 \\
$40$      &$7.50E{-2}$ & 1.81    &$6.37E{-2}$ & 1.85    &$4.53E{-0}$ & 1.00
          &$5.21E{-2}$ & 1.84    &$5.72E{-2}$ & 1.96    &$4.53E{-0}$ & 1.00 \\
$80$      &$2.22E{-2}$ & 1.76    &$1.87E{-2}$ & 1.77    &$2.27E{-0}$ & 1.00
          &$1.40E{-2}$ & 1.90    &$1.45E{-2}$ & 1.98    &$2.26E{-0}$ & 1.00 \\
$160$     &$9.12E{-3}$ & 1.28    &$6.06E{-3}$ & 1.63    &$1.13E{-0}$ & 1.00
          &$3.62E{-3}$ & 1.95    &$3.64E{-3}$ & 1.99    &$1.13E{-0}$ & 1.00 \\
$320$     &$4.07E{-3}$ & 1.16    &$2.22E{-3}$ & 1.45    &$5.66E{-1}$ & 1.00
          &$9.24E{-4}$ & 1.97    &$9.14E{-4}$ & 2.00    &$5.66E{-1}$ & 1.00 \\
\hline
\end{tabular}
\caption{Errors of nonsymmetric DG-IFE solutions with $\beta^- = 10$, $\beta^+ = 1$}
\label{table: nonsym DGIFE error 10 1}
\end{footnotesize}
\end{center}
\end{table}

\begin{table}[ht]
\begin{center}
\begin{footnotesize}
\begin{tabular}{|c|cc||cc||cc||cc|cc||cc|}
\hline
& \multicolumn{6}{|c||}{\textbf{Backward Euler}}&  \multicolumn{6}{|c|}{\textbf{Crank Nicolson}}\\
\hline
$N_s$
& $\|\cdot\|_{L^\infty}$ & rate
& $\|\cdot\|_{L^2}$ & rate
& $|\cdot|_{H^1}$ & rate
& $\|\cdot\|_{L^\infty}$ & rate& $\|\cdot\|_{L^2}$ & rate
& $|\cdot|_{H^1}$ & rate  \\
 \hline
$10$      &$7.23E{-1}$ &         &$5.56E{-1}$ &         &$1.81E{+1}$ &
          &$9.04E{-1}$ &         &$5.61E{-1}$ &         &$1.81E{+1}$ &      \\
$20$      &$2.18E{-1}$ & 1.73    &$1.41E{-1}$ & 1.98    &$9.07E{-0}$ & 1.00
          &$1.79E{-1}$ & 2.34    &$1.38E{-1}$ & 2.02    &$9.06E{-0}$ & 1.00 \\
$40$      &$6.21E{-2}$ & 1.82    &$3.88E{-2}$ & 1.86    &$4.53E{-0}$ & 1.00
          &$5.54E{-2}$ & 1.69    &$3.35E{-2}$ & 2.04    &$4.53E{-0}$ & 1.00 \\
$80$      &$1.75E{-2}$ & 1.83    &$1.20E{-2}$ & 1.69    &$2.26E{-0}$ & 1.00
          &$1.60E{-2}$ & 1.80    &$8.24E{-3}$ & 2.02    &$2.26E{-0}$ & 1.00 \\
$160$     &$5.93E{-3}$ & 1.56    &$4.33E{-3}$ & 1.48    &$1.13E{-0}$ & 1.00
          &$4.54E{-3}$ & 1.82    &$2.04E{-3}$ & 2.01    &$1.13E{-0}$ & 1.00 \\
$320$     &$3.26E{-3}$ & 0.86    &$1.78E{-3}$ & 1.28    &$5.66E{-1}$ & 1.00
          &$1.24E{-3}$ & 1.88    &$5.08E{-4}$ & 2.01    &$5.66E{-1}$ & 1.00 \\
\hline
\end{tabular}
\caption{Errors of symmetric DG-IFE solutions with $\beta^- = 10$, $\beta^+ = 1$}
\label{table: sym DGIFE error 10 1}
\end{footnotesize}
\end{center}
\end{table}

\subsubsection*{Example 3: Large Jump $(\beta^-,\beta^+) = (1,10000)$ and $(\beta^-,\beta^+) = (10000,1)$}
In this example we enlarge the contrast of the diffusion coefficient such that $(\beta^-,\beta^+) = (1,10000)$, and $(\beta^-,\beta^+) = (10000,1)$ . Here we use  nonsymmetric DG-IFE scheme and the penalty parameter is chosen as $\sigma_e = 1$. Data listed in Table \ref{table: nonsym DGIFE error 1 10000} and Table \ref{table: nonsym DGIFE error 10000 1} are generated with time step size $\Delta t = 2h$.

\begin{table}[htb]
\begin{center}
\begin{footnotesize}
\begin{tabular}{|c|cc||cc||cc||cc|cc||cc|}
\hline
& \multicolumn{6}{|c||}{\textbf{Backward Euler}}&  \multicolumn{6}{|c|}{\textbf{Crank Nicolson}}\\
\hline
$N_s$
& $\|\cdot\|_{L^\infty}$ & rate
& $\|\cdot\|_{L^2}$ & rate
& $|\cdot|_{H^1}$ & rate
& $\|\cdot\|_{L^\infty}$ & rate& $\|\cdot\|_{L^2}$ & rate
& $|\cdot|_{H^1}$ & rate  \\
 \hline
$10$      &$1.39E{-1}$ &         &$3.06E{-2}$ &         &$1.08E{-0}$ &
          &$1.80E{-1}$ &         &$3.22E{-2}$ &         &$1.10E{-0}$ &      \\
$20$      &$5.07E{-2}$ & 1.46    &$8.67E{-3}$ & 1.82    &$5.68E{-1}$ & 0.93
          &$3.67E{-2}$ & 2.29    &$8.89E{-3}$ & 1.86    &$5.64E{-1}$ & 0.97 \\
$40$      &$1.36E{-2}$ & 1.89    &$2.33E{-3}$ & 1.89    &$2.94E{-1}$ & 0.95
          &$1.08E{-2}$ & 1.76    &$2.35E{-3}$ & 1.92    &$2.93E{-1}$ & 0.95 \\
$80$      &$3.64E{-3}$ & 1.90    &$6.16E{-3}$ & 1.92    &$1.49E{-1}$ & 0.98
          &$3.24E{-3}$ & 1.74    &$6.04E{-3}$ & 1.96    &$1.49E{-1}$ & 0.98 \\
$160$     &$1.03E{-3}$ & 1.82    &$1.63E{-4}$ & 1.92    &$7.50E{-2}$ & 0.99
          &$9.52E{-4}$ & 1.77    &$1.50E{-4}$ & 2.01    &$7.49E{-2}$ & 0.99\\
$320$     &$2.78E{-4}$ & 1.89    &$4.68E{-5}$ & 1.80    &$3.76E{-2}$ & 0.99
          &$2.59E{-4}$ & 1.88    &$3.86E{-5}$ & 1.96    &$3.76E{-2}$ & 0.99 \\
\hline
\end{tabular}
\caption{Errors of nonsymmetric DG-IFE solutions with $\beta^- = 1$, $\beta^+ = 10000$}
\label{table: nonsym DGIFE error 1 10000}
\end{footnotesize}
\end{center}
\end{table}

\begin{table}[htb]
\begin{center}
\begin{footnotesize}
\begin{tabular}{|c|cc||cc||cc||cc|cc||cc|}
\hline
& \multicolumn{6}{|c||}{\textbf{Backward Euler}}&  \multicolumn{6}{|c|}{\textbf{Crank Nicolson}}\\
\hline
$N_s$
& $\|\cdot\|_{L^\infty}$ & rate
& $\|\cdot\|_{L^2}$ & rate
& $|\cdot|_{H^1}$ & rate
& $\|\cdot\|_{L^\infty}$ & rate& $\|\cdot\|_{L^2}$ & rate
& $|\cdot|_{H^1}$ & rate  \\
 \hline
$10$      &$9.36E{-1}$ &         &$8.33E{-1}$ &         &$1.82E{+1}$ &
          &$1.19E{-0}$ &         &$8.61E{-1}$ &         &$1.83E{+1}$ &      \\
$20$      &$2.64E{-1}$ & 1.83    &$2.26E{-1}$ & 1.88    &$9.08E{-0}$ & 1.00
          &$1.87E{-1}$ & 2.67    &$2.20E{-1}$ & 1.97    &$9.06E{-0}$ & 1.01 \\
$40$      &$7.50E{-2}$ & 1.81    &$6.27E{-2}$ & 1.85    &$4.53E{-0}$ & 1.00
          &$5.21E{-2}$ & 1.84    &$5.64E{-2}$ & 1.96    &$4.53E{-0}$ & 1.00 \\
$80$      &$2.17E{-2}$ & 1.79    &$1.83E{-2}$ & 1.78    &$2.27E{-0}$ & 1.00
          &$1.40E{-2}$ & 1.90    &$1.43E{-2}$ & 1.98    &$2.26E{-0}$ & 1.00 \\
$160$     &$8.92E{-3}$ & 1.28    &$5.89E{-3}$ & 1.64    &$1.13E{-0}$ & 1.00
          &$3.62E{-3}$ & 1.95    &$3.59E{-3}$ & 1.99    &$1.13E{-0}$ & 1.00\\
$320$     &$3.99E{-3}$ & 1.16    &$2.14E{-3}$ & 1.46    &$5.66E{-1}$ & 1.00
          &$9.24E{-4}$ & 1.97    &$8.99E{-5}$ & 2.00    &$5.66E{-1}$ & 1.00 \\
\hline

\end{tabular}
\caption{Errors of nonsymmetric DG-IFE solutions with $\beta^- = 10000$, $\beta^+ = 1$}
\label{table: nonsym DGIFE error 10000 1}
\end{footnotesize}
\end{center}
\end{table}

For all examples above, we also experimented linear IFE functions on structured triangular meshes, which is formed by cutting each rectangle of $\mathcal{T}_h$ into two triangles.  The numerical results are very similar to the rectangular meshes; hence, we omit the data in this article.

\section{Conclusion}
In this article, we developed a class of discontinuous Galerkin scheme for solving parabolic interface problem. Taking advantages of immersed finite element functions, the proposed methods can be used on Cartesian mesh regardless of the location of interface. \textit{A priori} error estimation shows that these DG-IFE methods converge to exact solution with an optimal order in the energy norm.

\bibliographystyle{abbrv}

\end{document}